\documentclass[10pt, article]{amsart}

\usepackage{tikz}
\usetikzlibrary{calc}
\usepackage{ae} % or {zefonts}
\usepackage[T1]{fontenc}
\usepackage[cp1250]{inputenc}
\usepackage{amsmath}
\usepackage{amssymb, amsfonts,amscd,verbatim}

\usepackage[normalem]{ulem}
\usepackage{hyperref}
\usepackage{indentfirst}
\usepackage{latexsym}
\input xy
\xyoption{all}

\usepackage{amsmath}    % need for subequations
\theoremstyle{plain}
\newtheorem{Pocz}{Poczatek}[section]
\newtheorem{Proposition}[Pocz]{Proposition}
\newtheorem{Theorem}[Pocz]{Theorem}
\newtheorem{Corollary}[Pocz]{Corollary}

\newtheorem{Lemma}[Pocz]{Lemma}
\newtheorem{Observation}[Pocz]{Observation}

\newtheorem{Example}[Pocz]{Example}

\theoremstyle{definition}
\newtheorem{Definition}[Pocz]{Definition}

\theoremstyle{remark}
\newtheorem{Remark}[Pocz]{Remark}

\DeclareMathOperator*{\diam}{diam}

\def\diam{\mathrm{diam}}

\numberwithin{equation}{section}

\title[
Unifying large scale and small scale geometry
]%
  {Unifying large scale and small scale geometry}

\author{Jerzy ~Dydak}
\address{University of Tennessee, Knoxville, USA}
\email{jdydak@utk.edu}

\date{ \today
}


\begin{document}
\maketitle
%\begin{center}
%\today
%\end{center}

\tableofcontents

\begin{abstract}
A topology on a set $X$ is the same as a projection (i.e. an idempotent linear operator)
$cl:2^X\to 2^X$ satisfying $A\subset cl(A)$ for all $A\subset X$. That's a good way to summarize Kuratowski's closure operator.

Basic geometry on a set $X$ is a dot product $\cdot:2^X\times 2^X\to 2^Y$. Its equivalent form is an orthogonality relation on subsets of $X$. The optimal case is if the orthogonality relation satisfies a variant of parallel-perpendicular decomposition from linear algebra.

We show that this concept unifies small scale (topology, proximity spaces, uniform spaces) and large scale (coarse spaces, large scale spaces). Using orthogonality relations we define large scale compactifications that generalize all well-known compactifications:
Higson corona, Gromov boundary, \v Cech-Stone compactification, Samuel-Smirnov compactification, and Freudenthal compactification.
\end{abstract}

\section{Dot products and orthogonality relations}

%beginComment
The common wisdom is that large scale and small scale are dual. We will show that, at certain level, they can be explained using the same structure/concept.

That concept is a \textbf{basic dot product} on subsets of a set $X$. 
Equivalently, it is an \textbf{orthogonality relation} on subsets of a set $X$. 
%endComment

\subsection{Dot products on sets}
\begin{Definition}
A \textbf{dot product} on a set $X$ is a symmetric function 
$$\cdot:2^X\times 2^X\to 2^Y$$
that is bi-linear in the following sense:\\
1. $\emptyset\cdot X=\emptyset$,\\
2. $C\cdot (D\cup E)=(C\cdot D)\cup (C\cdot E)$ for all subsets $C, D, E$ of $X$.

A \textbf{basic dot product} on $X$ is one that admits only two values: $\emptyset$ and $Y$.
\end{Definition}

\begin{Observation}
Notice every dot product on $X$ can be reduced to a basic dot product by changing all the non-empty values to $X$.

\end{Observation}

\begin{Observation}
If one believes that the values of a dot product should be scalars, then there is a way to interpret basic dot products to adhere to that belief. Namely, scalars for $X$ are $\emptyset$ and $X$. The scalar multiplication $c\cdot A$ is $c\cap A$ and a basic dot product $\cdot$ is a function $\odot$ from $2^X\times 2^X$ to scalars of $X$ such that
$$A\odot (c_1\cdot B\cup c_2\cdot D)=c_1\cdot (A\odot B)\cup c_2\cdot (A\odot D).$$
\end{Observation}

\begin{Example}
1. Every set $X$ has a natural dot product defined by
$$C\cdot D=C\cap D.$$
2. Every topology on $X$ induces its dot product defined by
$$C\cdot D=cl(C)\cap cl(D).$$
3. Every subset $X$ of a topological space $\bar X$ has the induced 
dot product defined by
$$C\cdot D=cl(C)\cap cl(D)\setminus X,$$
where the closures are taken in $\bar X$.
\end{Example}

\subsection{Orthogonality relations on sets}
The only information a basic dot product $\cdot$ carries is which sets $C, D$ are $\cdot$-orthogonal, i.e. $C\cdot D=\emptyset$. Therefore, it makes sense to define the relation of orthogonality axiomatically.

%beginComment
\begin{Definition}
An \textbf{orthogonality relation} on subsets of a set $X$ is a symmetric relation $\perp$ satisfying the following properties:\\
1. $\emptyset \perp X$,\\
2. $A\perp (C\cup C')\iff A\perp C$ and $A\perp C'$.
\end{Definition}

%endComment

\begin{Observation}
One can reduce the number of axioms by dropping symmetry and replacing Axiom 2 by\\
2'. $A\perp (C\cup C')\iff C\perp A$ and $C'\perp A$.
\end{Observation}

\begin{Example}\label{OrthInducedByBornology}
For every bornology $\mathcal{B}$ on a set $X$ the relation $A\perp C$ defined as $A\cap C\in \mathcal{B}$ is an orthogonality relation.
\end{Example}
\begin{proof}
Recall that a bornology on $X$ is any family of subsets closed under finite unions 
so that $B\subset B'\in\mathcal{B}$ implies $B\in\mathcal{B}$.
\end{proof}

\begin{Proposition}
Suppose $X$ is a set.\\
1. Every dot product $\cdot$ on $X$ induces an orthogonal relation on $X$ defined by
$$C\perp D\iff C\cdot D=\emptyset.$$
2. Every orthogonality relation $\perp$ on $X$ induces a basic dot product $\cdot$
defined as follows: $C\cdot D=\emptyset$, if $C\perp D$, $C\cdot D=X$ otherwise.
\end{Proposition}
\begin{proof}
1. $\emptyset\perp X$ as $\emptyset\cdot X=\emptyset$.
If $C\subset C'$ and $C'\perp D$, then $\emptyset=C'\cdot D=(C\cup (C'\setminus C))\cdot D=C\cdot D\cup (C'\setminus C)\cdot D$ resulting in $C\cdot D=\emptyset$.
If $C\perp D$ and $C'\perp D$, then $\emptyset=C\cdot D\cup C'\cdot D=(C\cup C')\cdot D$ resulting in $D\perp (C\cup C')$.\\
2. Left to the reader.
\end{proof}

%beginComment
\subsection{Examples of small scale orthogonality}
\begin{Example}
1. \textbf{Set-theoretic orthogonality}: Disjointness,\\
2. \textbf{Topological orthogonality}: Disjointness of closures,\\
3. \textbf{Metric orthogonality}: Disjointness of $r$-balls for some $r > 0$,\\
4. \textbf{Uniform orthogonality}: Disjointness of $\mathcal{U}$-neighborhoods for some uniform cover $\mathcal{U}$.
\end{Example}

%endComment

%beginComment
\subsection{Examples of large scale orthogonality}
\begin{Example}
1. \textbf{Set-theoretic large scale orthogonality}: Finiteness of intersection,\\
2. \textbf{Metric large scale orthogonality}: Boundedness of intersection of $r$-balls for all $r > 0$,\\
3. \textbf{Group large scale orthogonality}: Finiteness of $(A\cdot F)\cap (C\cdot F)$ for all finite subsets $F$ of a group $G$.

Same as metric ls-orthogonality for word metrics if $G$ is finitely generated.\\
4. \textbf{Topological ls-orthogonality}: Disjointness of coronas of closures in a fixed compactification $\bar X$ of $X$.
\end{Example}

%endComment

%beginComment
\subsection{Hyperbolic orthogonality}
Given a metric space $(X,d)$, the \textbf{Gromov product} of $x$ and $y$ with respect to $a\in X$ is defined by
$$
\left< x,y\right>_a=\frac{1}{2}\big(d(x,a)+d(y,a)-d(x,y)\big).
$$

Recall that metric space $(X,d)$ is (Gromov) $ \delta-$\textbf{hyperbolic} if it satisfies the $\delta/4$-inequality:
$$
\left< x,y\right>_{a} \geq \min \{\left< x,z\right>_{a},\left< z,y \right> _{a}\}-\delta/4, \quad \forall x,y,z,a\in X.$$

$(X,d)$ is \textbf{Gromov hyperbolic} if it is $ \delta-$hyperbolic for some $\delta > 0$.
%endComment

%beginComment
\begin{Definition}
Two subsets $A$ and $C$ of a hyperbolic space $X$ are \textbf{hyperbolically orthogonal} if there is $r > 0$ such that
$$\left< a,c\right>_{p} < r$$
for some fixed $p$ and all $(a,c)\in A\times C$.
\end{Definition}

%endComment

%beginComment
\subsection{Freundenthal orthogonality}
\begin{Definition}
Suppose $X$ is a locally compact and locally connect topological space.
Two subsets $A$ and $C$ of $X$ are \textbf{Freundenthal orthogonal}
if there is a compact subset $K$ of $X$ such that the union of all components of $X\setminus K$ intersecting $A$ is disjoint from the union of all components of $X\setminus K$ intersecting $C$.
\end{Definition}

%endComment

%beginComment
\subsection{Bounded sets}
\begin{Definition}
Given an orthogonality relation $\perp$ on subsets of $X$, a \textbf{bounded subset} $B$ of $X$ is one that is orthogonal to the whole set: 
$$B\perp X.$$
\end{Definition}

%endComment

%beginComment
 \begin{Definition}
An orthogonality relation $\perp$ on subsets of $X$ is \textbf{small scale} if the empty set is the only  subset of $X$ that is orthogonal to itself. In particular, the only bounded subset of $X$ is the empty set.
\end{Definition}

\begin{Definition}
An orthogonality relation $\perp$ on subsets of $X$ is \textbf{large scale} if each point is a bounded subset of $X$.
\end{Definition}

%endComment

\subsection{Normal orthogonal relations}

\begin{Definition}\label{SpanDef}
Given an orthogonality relation $\perp$ on subsets of $X$, two subsets $C$ and $D$ \textbf{$\perp$-span} $X$ if the following conditions are satisfied:\\
1. $C\perp D$,\\
2. $X$ can be decomposed as $X=C'\cup D'$, where $C'\perp D$ and $D'\perp C$.
\end{Definition}

\begin{Remark}
Obviously, we may interpret the word "decompose" in the definition above as $C'\cap D'=\emptyset$ since $D'$ can be replaced by $X\setminus C'$.
The other extreme is when $C\subset C'$ and $D\subset D'$ which can be accomplished by replacing $C'$ by $C\cup C'$ and replacing $D'$ by $D\cup D'$. In that case we may think of $C$ being parallel to $C'$, $D$ being parallel to $D'$ and interpret Definition \ref{SpanDef} as an analog of \textbf{parallel-perpendicular decomposition} in Linear Algebra.
\end{Remark}

%beginComment
\begin{Definition}
An orthogonality relation $\perp$ on subsets of $X$ is \textbf{Fr\' echet} if
$\{x\}\perp \{y\}$ whenever $x, y\in X$ and $x\ne y$.
\end{Definition}
%endComment

%beginComment
\begin{Definition}
An orthogonality relation $\perp$ on subsets of $X$ is \textbf{Hausdorff} if
$\{x\}$ and $\{y\}$ $\perp$-span $X$ whenever $x, y\in X$ and $x\ne y$.
\end{Definition}
%endComment

%beginComment
\begin{Definition}
An orthogonality relation $\perp$ on subsets of $X$ is \textbf{regular} (or \textbf{Vietoris}) if\\
1. it is Fr\' echet,\\
2. $\{x\}$ and $A$ $\perp$-span $X$ whenever $x\perp A$.\\
3. $x\perp x$ implies $x$ is $\perp$-bounded.
\end{Definition}
%endComment

%beginComment
\begin{Definition}
An orthogonality relation $\perp$ on subsets of $X$ is \textbf{normal} (or \textbf{Tietze}) if\\
1. it is Fr\' echet,\\
2. $C$ and $D$ $\perp$-span $X$ whenever $C\perp D$.\\
3. $B\perp B$ implies $B$ is $\perp$-bounded.
\end{Definition}
%endComment

%beginComment
\begin{Example}
1. The topological orthogonality relation on a topological space is Hausdorff if and only if $X$ is topologically Hausdorff.\\
2. The topological orthogonality relation on a topological space is regular if and only if $X$ is topologically regular.\\
3. The topological orthogonality relation on a topological space is normal if and only if $X$ is topologically normal.
\end{Example}

%endComment

\begin{Definition}
The \textbf{functional orthogonality relation} $\perp$ on a topological space $X$ is defined as follows: $C\perp D$ if there is a continuous function $f:X\to [0,1]$ such that
$f(C)\subset \{0\}$ and $f(D)\subset \{1\}$.
\end{Definition}

\begin{Proposition}
The functional orthogonality relation $\perp$ on a topological space $X$ is normal if and only if $X$ is functionally Hausdorff.
\end{Proposition}
\begin{proof}
$\{x\}\perp \{y\}$, where $x, y\in X$ and $x\ne y$, means there is a continuous function
$f:X\to [0,1]$ such that $f(x)=0$ and $f(x)=1$. That is precisely the definition of being functionally Hausdorff. Also, for any continuous $f:X\to [0,1]$ satisfying
$f(C)\subset \{0\}$, $f(D)\subset \{1\}$, one puts $C'=f^{-1}[0.5,1]$, $D'=f^{-1}[0,0.5]$
and observe $X=C'\cup D'$, $C'\perp C$, and $D'\perp D$.
\end{proof}

\section{Topology induced by orthogonal relations}

%beginComment
There are at least two topologies induced by orthogonality relations. The most useful is the one based on the following concept:
\begin{Definition}\label{PerpADef}
Given an orthogonality relation $\perp$ on $X$ and $A\subset X$, $A^\perp$ is defined
as 
$$A^\perp :=\{x\in X\setminus A\mid x\perp A\}.$$
\end{Definition}

\begin{Proposition}
$C^\perp\cap D^\perp=(C\cup D)^\perp$.
\end{Proposition}
\begin{proof}
Left to the reader.
\end{proof}

\begin{Definition}
Given an orthogonality relation $\perp$ on $X$, the \textbf{topology induced} by $\perp$ has $\{A^\perp\mid A\subset X\}\cup \{\emptyset\}$ as its basis.
\end{Definition}

%endComment

\begin{Example}
Suppose $\mathcal{B}$ is a non-empty bornology on $X$ and $\perp$ is the orthogonality relation induced by
$\mathcal{B}$ (see \ref{OrthInducedByBornology}).
The topology induced by $\perp$ is discrete.
\end{Example}
\begin{proof}
Since $A\cap C=\emptyset\in \mathcal{B}$ implies $A\perp C$, 
$(X\setminus \{x\})^\perp=\{x\}$ for all $x\in X$ and all subsets of $X$ are open.
\end{proof}

%beginComment
\begin{Example}
Large scale orthogonal relations induce discrete topologies.
\end{Example}

\begin{Proposition}
Suppose $\perp$ is an orthogonality relation on a set $X$. If $X$ is Hausdorff, then for each two different points $x,y\in X$ there are subsets $C, D$ of $X$ such that $x\in C^\perp$, $y\in D^\perp$, $x\perp D^\perp$, $y\perp C^\perp$,
and $C^\perp\cap D^\perp=\emptyset$.
\end{Proposition}
\begin{proof}
Pick two disjoint sets, $C$ containing $y$ and $D$ containing $x$, such that
$x\perp C$, $y\perp D$, and $C\cup D=X$. Notice
$x\in C^\perp\subset D$ and $y\in D^\perp\subset C$, so $C^\perp\cap D^\perp=\emptyset$.
\end{proof}

\begin{Proposition}
Suppose $\perp$ is an orthogonality relation on a set $X$. If $X$ is regular, then for each subset $A\subset X$ and each point $x\notin A$ there are subsets $C, D$ of $X$ such that $x\in C^\perp$, $A\subset D^\perp$, $x\perp D^\perp$, $A\perp C^\perp$,
and $C^\perp\cap D^\perp=\emptyset$.
\end{Proposition}
\begin{proof}
Pick two disjoint sets, $C$ containing $A$ and $D$ containing $x$, such that
$x\perp C$, $A\perp D$, and $C\cup D=X$. Notice
$x\in C^\perp\subset D$ and $A\subset D^\perp\subset C$, so $C^\perp\cap D^\perp=\emptyset$.
\end{proof}

There is another way to define a topology on $X$ given an orthogonality relation $\perp$: $A$ is \textbf{closed} if $x\perp A$ for all $x\notin A$. In the case of regular relations those two topologies coincide.

\begin{Proposition}
If $\perp$ is a regular orthogonality relation on a set $X$, then
$$(X\setminus A^\perp)^\perp =A^\perp$$
for all subsets $A$ of $X$.
\end{Proposition}
\begin{proof}
It suffices to show $A^\perp\subset (X\setminus A^\perp)^\perp$.
Suppose $x\in A^\perp$, i.e. $x\in X\setminus A$ and $x\perp A$.
Therefore there is $X'\subset X\setminus A$ containing $x$ such that $x\perp X\setminus X'$
and $X'\perp A$. Hence $X'\subset A^\perp $ resulting in $X\setminus A^\perp\subset X\setminus X'$ and, consequently, $x\perp X\setminus A^\perp$.

\end{proof}

\begin{Proposition}
The topology induced by a functional orthogonality relation $\perp$ on a Hausdorff space $(X,\mathcal{T})$ equals $\mathcal{T}$  if and only if $X$ is completely regular (Tychonoff).
\end{Proposition}
\begin{proof}
Being completely regular means exactly that $\{x\}\perp \{y\}$ if $x\ne y$ and that
$\{x\}\perp X\setminus U$ if $U$ is open and $x\in U$.
\end{proof}

\begin{Proposition}\label{MainPropOnNormalRelations}
If $\perp$ is a normal orthogonal relation on a set $X$ and $C\perp D$, then there exist subsets $E$ and $F$ of $X$ such that $C\subset E^\perp$, $D\subset F^\perp$ and
$E^\perp \perp F^\perp$.
\end{Proposition}
\begin{proof}
Notice $B=C\cap D$ is $\perp$-bounded. Find disjoint sets $C'$ containing $C\setminus B$ and $D'$ containing $D\setminus B$ whose union is $X$ and $C'\perp D$,
$D'\perp C$. Put $E=D'\setminus B$ and notice $C\subset E^\perp\subset C'\cup B$,
so $E^\perp \perp D$. Repeat the same procedure to create $F$.
\end{proof}

\section{Proximity spaces}
There is a more general structure than uniform spaces, namely a proximity (see \cite{NaWa}). In this section we show that those structures correspond to normal small scale orthogonal relations.
\begin{Definition}
A \textbf{proximity space} $(X, \delta)$ is a set $X$ with a relation $\delta$ between subsets of $X$ satisfying the following properties:\\
For all subsets $A, B$ and $C$ of $X$\\
1. $A \delta B \implies B \delta A$\\
2. $A \delta B \implies A \ne\emptyset$\\
3. $A\cap B\ne\emptyset \implies A \delta B$\\
4. $A \delta (B\cup C) \iff (A \delta B \mbox{ or } A \delta C)$\\
5. $\forall E, A \delta E \mbox{ or }B \delta (X\setminus E) \implies A \delta B$.
\end{Definition}

\begin{Proposition}
Normal small scale orthogonal relations are in one-to-one correspondence with proximity relations.
\end{Proposition}
\begin{proof}
Given a small scale orthogonal relation $\perp$ we define $A\delta C$ as $\lnot (A\perp C)$.

Conversely, given a proximity relation $\delta$ we define $A\perp C$ as $\lnot (A\delta C)$.

The proof amounts to negating implications, so let's show only the implication $A\cap B\ne\emptyset \implies A \delta B$. If it fails, then we have two orthogonal sets $A$ and $B$ with non-empty intersection $A\cap B$. However, in this case $A\cap B$ is self-orthogonal, a contradiction.
\end{proof}

\section{Asymptotic resemblance}

S. Kalantari and B. Honari \cite{KalHon} introduced an equivalence relation $\lambda$
between subsets of a set $X$ called asymptotic resemblance.
In this section we show that, under natural condition of all points of $X$ being equivalent,
each asymptotic resemblance induces an orthogonal relation.

\begin{Definition}
\textbf{Asymptotic resemblance} $\lambda$
between subsets of a set $X$ is an equivalence relation satisfying the following properties:\\
1. $A_1\lambda B_1$ and $A_2\lambda B_2$ implies $(A_1\cup A_2)\lambda (B_1\cup B_2)$.\\
2. $A\lambda (B_1\cup B_2)$ and $B_1, B_2\ne\emptyset$ implies existence
of non-empty subsets $A_1, A_2$ of $A$ such that $A=A_1\cup A_2$,
$A_1\lambda B_1$, and $A_2\lambda B_2$.
\end{Definition}

\begin{Proposition}
If $\lambda$ is an asymptotic resemblance relation on subsets of $X$ such that
$x\lambda y$ for all $x, y\in X$, then the relation $A\perp C$ defined using the three steps below is an orthogonal relation.\\
1. First, we define $A\le C$ as $C\lambda (A\cup C)$.\\
2. Second, we define $B$ to be bounded if $B\le A$ for all $A\subset X$, $A\ne\emptyset$.\\
3. Third, we define $A\perp C$ if $B\le A$ and $B\le C$ implies $B$ is bounded.
\end{Proposition}
\begin{proof}
Notice $C'\subset C$ implies $C'\le C$ and $A\le C$, $C\le D$ implies $A\le D$.
Consequently, $A\perp C$ and $C'\subset C$ implies $A\perp C'$.

Assume $A\perp C_1$ and $A\perp C_2$. If not $A\perp (C_1\cup C_2)$,
then there is an unbounded set $D$ such that
$A\lambda (A\cup D)$ and $(C_1\cup C_2)\lambda (C_1\cup C_2\cup D)$.
Now, we can split $D$ as $D_1\cup D_2$ so that $D_1\le C_1$ and $D_2\le C_2$.
Therefore, both $D_1$ and $D_2$ are bounded resulting in $D=D_1\cup D_2$ being bounded, a contradiction.
\end{proof}

\section{Morphisms}

%beginComment

\begin{Definition}
Given two sets $X$ and $Y$ equipped with orthogonality relations $\perp_X$ and $\perp_Y$, a function $f:X\to Y$ is \textbf{$\perp$-continuous}
if 
$$A\perp_Y C\implies f^{-1}(A)\perp_X f^{-1}(C)$$
for all subsets $A, C$ of $Y$.
\end{Definition}

%endComment

%beginComment
\subsection{Small Scale Examples}
In the small scale $\perp$-continuous functions are exactly neighborhood-continuous functions with respect to the induced neighborhood operator. Therefore both examples below follow from \cite{DW} in view of \ref{PerpContinuousVsNbhdCont}.
\begin{Example}
If both $X$ and $Y$ are normal spaces equipped with topological orthogonality relations, then $\perp$-continuity is ordinary \textbf{topological continuity}.
\end{Example}

\begin{Example}
If both $X$ and $Y$ are uniform spaces equipped with uniform orthogonality relations, then $\perp$-continuity is ordinary \textbf{uniform continuity}.
\end{Example}

%beginComment
\subsection{Large Scale Examples}
\begin{Example}
If both $X$ and $Y$ are metric spaces equipped with metric $ls$-orthogonality relations
and $f:X\to Y$ preserves bounded sets, then $\perp$-continuity is the same as $f$ being \textbf{coarse and bornologous}.
\end{Example}
\begin{proof}
Recall that $f:X\to Y$ is bornologous if, for each $r > 0$, there is $s > 0$ such that
$\diam(f(A)) < s$ if $\diam(A) < r$. 

Notice that every $\perp$-continuous function co-preserves bounded sets, so it is coarse.
Suppose $f$ is $\perp$-continuous but not bornologous. Hence, there is a sequence
$B_n$ of uniformly bounded subsets of $X$ whose images $f(B_n)$ have diameters diverging to infinity. We may reduce it to the case of each $B_n$ consisting of exactly two points $x_n$ and $y_n$ so that both $f(x_n)$ and $f(y_n)$ diverge to infinity.
Notice $A:=\{f(x_n)\}_{n\ge 1}$ and $C:=\{f(y_n)\}_{n\ge 1}$ are orthogonal in $Y$
but their point-inverses are not orthogonal in $X$, a contradiction.

Suppose $f$ is coarse and bornologous but not $\perp$-continuous. Choose two orthogonal subsets
$A$ and $C$ of $Y$ whose point-inverses are not orthogonal.
Therefore the intersection of $B(f^{-1}(A),r)$ and $B(f^{-1}(C),r)$ is unbounded
for some $r > 0$ and the image of that intersection is unbounded. There is $s > 0$ satisfying
$f(B(Z,r))\subset B(f(Z),s)$ for all subsets $Z$ of $X$.
Therefore, the intersection of $B(A,s)$ and $B(C,s)$ is unbounded, a contradiction.
\end{proof}

\begin{Example}
If $X$ is a metric space equipped with metric ls-orthogonality relation and $Y$ is a compact metric space equipped with small scale metric orthogonality, then $\perp$-continuity is the same as $f$ being \textbf{slowly oscillating}.
\end{Example}
\begin{proof}
Recall that $f:X\to Y$ is slowly oscillating if, for every pair of sequences $\{x_n\}_{n\ge 1}$, $\{y_n\}_{n\ge 1}$ in $X$, $\lim\limits_{n\to\infty}d_Y(f(x_n),f(y_n))=0$
if $\{d_X(x_n,y_n)\}_{n\ge 1}$ is uniformly bounded.

Suppose $f$ is $\perp$-continuous but not slowly oscillating. Hence, there is 
pair of sequences $\{x_n\}_{n\ge 1}$, $\{y_n\}_{n\ge 1}$ in $X$, 
and $\epsilon > 0$ such that $d_Y(f(x_n),f(y_n)) > \epsilon$ for each $n\ge 1$
and $\{d_X(x_n,y_n)\}_{n\ge 1}$ is uniformly bounded.
We may assume that the limit of $f(x_n)$ is $z_1$, the limit of $f(y_n)$ is $z_2$.
In particular $d_Y(z_1,z_2)\ge \epsilon$.
The sets $B(z_1,\epsilon/3)$ and $B(z_2,\epsilon/3)$ are orthogonal in $Y$
but their point-inverses in $X$ are not, a contradiction.

Suppose $f$ is slowly oscillating but not $\perp$-continuous. Choose two orthogonal subsets
$A$ and $C$ of $Y$ whose point-inverses are not orthogonal.
Therefore the intersection of $B(f^{-1}(A),r)$ and $B(f^{-1}(C),r)$ is unbounded
for some $r > 0$. 
Therefore there are two sequences diverging to infinity in $X$: $\{x_n\}_{n\ge 1}$
in $f^{-1}(A)$ and $\{y_n\}_{n\ge 1}$
in $f^{-1}(C)$ such that $d_X(x_n,y_n) < 2r$ for each $n$.
Consequently, $\lim\limits_{n\to\infty}d_Y(f(x_n),f(y_n))=0$
contradicting orthogonality of $A$ and $C$.
\end{proof}

\subsection{Quotient structures}

It is well-known that defining quotient maps in both the uniform category and in the coarse category is tricky. In contrast, in sets equipped with orthogonality relations it is quite easy.

\begin{Definition}
Suppose $\perp_X$ is an orthogonality relation on a set $X$. Given a surjective function $f:X\to Y$ define $C\perp_Y D$ to mean $f^{-1}(C)\perp_X f^{-1}(D)$.
\end{Definition}

It is easy to check that $\perp_Y$ is an orthogonality relation on $Y$, called the \textbf{quotient orthogonality relation}. Also, it is clear that the following holds:

\begin{Proposition}
Suppose $\perp_X$ is an orthogonality relation on a set $X$, $f:X\to Y$ is a surjective function, and $Y$ is equipped with the quotient orthogonality relation $\perp_Y$.
Given any $\perp$-continuous $h:X\to Z$ that is constant on fibers of $f$,
there is unique $\perp$-continuous $g:Y\to Z$ such that $h=g\circ f$.
\end{Proposition}

\section{Neighborhood operators}

This section is devoted to explore the relation between orthogonal relations and neighborhood operators.

\begin{Definition} \cite{DW}
A \textbf{neighborhood operator} $\prec$ on a set $X$ is a relation between its subsets satisfying the following conditions:

\begin{itemize}
\item[$\mathsf{(N0)}$] $A \prec X$ for all $A \subseteq X$.
\item[$\mathsf{(N1)}$] if $A \prec B$ then $X \setminus B \prec X \setminus A$.
\item[$\mathsf{(N2)}$] if $A \prec B \subseteq C$, then $A \prec C$.
\item[$\mathsf{(N3)}$] if $A \prec N$ and $A' \prec N'$ then $A \cup A' \prec N \cup N'$.
\end{itemize}
\end{Definition}

\begin{Observation}
Note that $\mathsf{(N0)}$ is implied by $\mathsf{(N1)}$ and the condition $X \prec X$. Also, it is easy to see that, together, axioms $\mathsf{(N0)}-\mathsf{(N3)}$ imply:

\begin{itemize}
\item[$\mathsf{(N0')}$] $\varnothing \prec A$ for all $A \subseteq X$.
\item[$\mathsf{(N2')}$] if $A \subseteq B \prec C$ then $A \prec C$.
\item[$\mathsf{(N3')}$] if $A \prec N$ and $A' \prec N'$ then $A \cap A' \prec N \cap N'$.
\end{itemize}
\end{Observation}

\begin{Definition}
A \textbf{normal neighborhood operator} $\prec$ satisfies the following condition:

\begin{itemize}
\item[$\mathsf{(N4)}$] for every pair of subsets $A \prec C$, there is a subset $B$ with $A \prec B \prec C$.
\end{itemize}
\end{Definition}

%beginComment
\begin{Proposition}
Each orthogonality relation $\perp$ on $X$ induces a neighborhood operator $\prec$ defined as follows: $A\prec U$ if $A\perp X\setminus U$ and $A\subset U$.\\
It is normal if and only if $\perp$ is normal.
\end{Proposition}
\begin{proof}
Left to the reader.
\end{proof}

%endComment

\begin{Proposition}
Each neighborhood operator $\prec$ on $X$ induces a small scale orthogonality relation $\perp$ 
defined as follows: $A\perp U$ if $A\prec X\setminus U$.\\
It is normal if and only if $\prec$ is normal.
\end{Proposition}
\begin{proof}
Left to the reader.
\end{proof}

\begin{Definition}\cite{DW} 
Let $X$ be a set and $\prec$ a neighborhood operator. If $A$ is a subset of $X$, then the \textbf{induced neighbourhood operator} $\prec_A$ on subsets of $A$ is defined as follows: $S \prec_A T$ precisely when there exists a subset $T'$ of $X$ such that $S \prec T'$ as subsets of $X$ and $T = T' \cap A$. 
\end{Definition}

\begin{Proposition}\label{PerpContinuousVsNbhdCont}
Suppose $X$ is a set equipped with an orthogonal relation $\perp_X$ and $Y$ is a set equipped with
a small scale orthogonality relation $\perp_Y$.
A function $f:A\subset X\to Y$ is neighborhood continuous (with respect to the induced neighborhood operators) if and only if it is $\perp$-continuous.
\end{Proposition}
\begin{proof}
Suppose $f:A\subset X\to Y$ is neighborhood continuous and $C\perp_Y D$.
Therefore $C\prec_Y Y\setminus D$ and $f^{-1}(C)\prec_A f^{-1}(Y\setminus D)$.
That means existence of $S\subset X$ such that $S\cap A= f^{-1}(Y\setminus D)$ and $ f^{-1}(C)\prec_X S$.
Consequently,  $f^{-1}(C)\perp_X (X\setminus S)$. Since  $f^{-1}(D)\subset X\setminus S$,
$ f^{-1}(D)\perp_X  f^{-1}(C)$.

Suppose $f:A\subset X\to Y$ is $\perp$-continuous and $C\prec_Y D$.
Hence $C\perp_Y (Y\setminus D)$ and $f^{-1}(C)\perp_X f^{-1}(Y\setminus D)$.
That implies $f^{-1}(C)\prec_X S$, where $S:=X\setminus f^{-1}(Y\setminus D)$.
Since $S\cap A=f^{-1}(D)$, $f$ is neighborhood continuous.
\end{proof}

\begin{Corollary}\label{RealExtensionLemma}
Suppose $X$ is a set equipped with a normal orthogonal relation $\perp_X$
and $[a,b]\subset \mathbb{R}$ is equipped with the topological orthogonality relation $\perp$.
If $f:A\subset X\to [a,b]$ is $\perp$-continuous, then it extends to a $\perp$-continuous
$\bar f:X\to [a,b]$. 
\end{Corollary}
\begin{proof}
In view of \ref{PerpContinuousVsNbhdCont}, it suffices to switch to neighborhood continuity
and that case is done in \cite{DW} (Theorem 8.5).
\end{proof}

\begin{Corollary}\label{ComplexExtensionLemma}
Suppose $X$ is a set equipped with a normal orthogonal relation $\perp_X$
and $\mathbb{C}$ is equipped with the topological orthogonality relation $\perp$.
If $f:A\subset X\to \mathbb{C}$ is $\perp$-continuous with metrically bounded image, then it extends to a $\perp$-continuous
$\bar f:X\to \mathbb{C}$ with metrically bounded image. 
\end{Corollary}
\begin{proof}
To apply \ref{RealExtensionLemma} it suffices to show that $g,h:A\to [a,b]$
are $\perp$-continuous if and only $g\Delta h:A\to [a,b]\times [a,b]$, $(g\Delta h)(x):=(g(x),h(x))$,
is $\perp$-continuous.

In one direction it is obvious, so assume $C, D\subset [a,b]\times [a,b]$ are metrically separated.
That means there is $\epsilon > 0$ such that $|z_1-z_2|\ge \epsilon$ if $z_1\in C$ and $z_2\in D$.
Cover $[a,b]\times [a,b]$
by finitely many sets of the form $B_1\times B_2$, where $B_1$ and $B_2$ are intervals of length $\epsilon/4$. Notice $(g\Delta h)^{-1}(C\cap (B_1\times B_2))\perp (g\Delta h)^{-1}(D\cap (B'_1\times B'_2))$ for any choice of $B_1, B_2, B_1', B_2'$. Therefore, $(g\Delta h)^{-1}(C\cap (B_1\times B_2))\perp (g\Delta h)^{-1}(D)$ for any choice of $B_1, B_2$.
Finally, $(g\Delta h)^{-1}(C)\perp (g\Delta h)^{-1}(D)$.
\end{proof}

\begin{Observation}\label{ProductOfPerpFunctions}
Observe that the proof of \ref{ComplexExtensionLemma} can be used to prove that, given two functions $f,g:X\to [0,1]$ from a set equipped with orthogonality relation $\perp$, the function $h:X\to [0,1]\times [0,1]$ is $\perp$-continuous if and only if both $f$ and $g$ are $\perp$-continuous.

\end{Observation}

\section{Simple parallelism structures}

In \cite{JDEnds} the concept of a simple coarse space was introduced. Now we can generalize it as follows:

\begin{Definition}
A \textbf{bounded structure} $\mathcal{B}$ on a set $X$ is a family of subsets of $X$ satisfying the following conditions:\\
1. $\{x\}\in \mathcal{B}$ for each $x\in X$,\\
2. $A\in \mathcal{B}$ if there is $C\in \mathcal{B}$ containing $A$,\\
3. $A\cup C\in \mathcal{B}$ if $A,C\in \mathcal{B}$ and $A\cap C\ne \emptyset$.

Elements of $\mathcal{B}$ are called \textbf{bounded subsets} of $X$.
\end{Definition}

\begin{Definition}\label{SimpleEndDef}
Suppose $(X,\mathcal{B})$ is a set $X$ equipped with a bounded structure $\mathcal{B}$.
A \textbf{simple end} in $(X,\mathcal{B})$ is a sequence $\{x_n\}_{n=1}^\infty$ in $X$ with the property that for any bounded set $A$ the set $\{n\in \mathbb{N} \mid x_n\notin A\}$ contains almost all natural numbers.
\end{Definition}

\begin{Definition}
Suppose $(X,\mathcal{B})$ is a set $X$ equipped with a bounded structure $\mathcal{B}$.
A \textbf{simple parallelism} on $(X,\mathcal{B})$ is an equivalence relation $\parallel$
on the set of simple ends of $X$ such that $\{x_n\}_{n\ge 1}\parallel \{y_n\}_{n\ge 1}$
implies $\{x_{a(n)}\}_{n\ge 1}\parallel \{y_{a(n)}\}_{n\ge 1}$ for all functions
$a:\mathbb{N}\to \mathbb{N}$ satisfying $\lim\limits_{n\to\infty}a(n)=\infty$.
\end{Definition}

\subsection{Small scale examples}
\begin{Example}
1. Any topological space $X$ whose bounded structure is empty
induces the simple parallelism defined as $\{x_n\}_{n\ge 1}\parallel \{y_n\}_{n\ge 1}$
if and only if $\{x_n\}_{n\ge 1}$ and $\{y_n\}_{n\ge 1}$ converge to the same point in $X$.\\
2. Any metric space $(X,d)$ whose bounded structure is empty
induces the simple parallelism defined as $\{x_n\}_{n\ge 1}\parallel \{y_n\}_{n\ge 1}$
if and only if $\lim\limits_{n\to\infty} d(x_n,y_n)=0$.\\
2. Any uniform space $X$ whose bounded structure is empty
induces the simple parallelism defined as $\{x_n\}_{n\ge 1}\parallel \{y_n\}_{n\ge 1}$
if and only if for any uniform cover $\mathcal{U}$ of $X$ there is $M > 0$
such that for each $n > M$ both $x_n$ and $y_n$ belong to the same element of $\mathcal{U}$.\\
\end{Example}

\subsection{Induced orthogonality relation}
\begin{Proposition}
Suppose $\parallel$ is a simple parallelism relation on a set $X$ equipped with a bornology $\mathcal{B}$.\\
1. $\parallel$ induces the orthogonality relation
$\perp$ defined as follows: $A\perp C$ if there are no simple ends
$\{x_n\}_{n\ge 1}$ in $A$ and $\{y_n\}_{n\ge 1}$ in $C$ that are parallel.\\
2. A subset $B$ of $X$ is $\perp$-bounded if and only if
it contains no simple end.\\
3. $\perp$ is a small scale orthogonality relation if and only if
$\mathcal{B}$ is empty.\\
4. $\perp$ is a large scale orthogonality relation if and only if
$\mathcal{B}$ contains all subsets of $X$ consisting of single point.
\end{Proposition}
\begin{proof}
1. Suppose $A\perp C$, $A\perp C'$ but $A\perp (C\cup C')$ fails.
In that case there are simple ends
$\{x_n\}_{n\ge 1}$ in $A$ and $\{y_n\}_{n\ge 1}$ in $C|cup C'$ that are parallel.
However, infinitely many elements of $\{y_n\}_{n\ge 1}$ are in one of $C$, $C'$, a contradiction.\\
2. If $B$ contains a simple end, then $B\perp B$ fails.\\
3. Is obvious.\\
4. Is obvious.
\end{proof}

\begin{Observation}
A set $D$ is closed in the topology induced by $\perp$ (in the case above) if and only if
for every sequence $\{x_n\}_{n\ge 1}$ in $D$ parallel to a constant sequence $\{c\}_{n\ge 1}$, $c$ is a point of $D$.
\end{Observation}

\section{Compactifications and orthogonality relations}

\begin{Proposition}\label{RelationsInducedByCompactifications}
Every compactification $\bar X$ of a locally compact Hausdorff space $X$ induces two orthogonal relations on $X$:\\
1. A small scale relation $\perp_{ss}$, where $A\perp_{ss} C$ means closures of $A$ and $C$ in $\bar X$ are disjoint.\\
2. A large scale relation $\perp_{ls}$, where $A\perp_{ls} C$ means closures of $A$ and $C$ in $\bar X$ are disjoint at $\bar X\setminus X$ (i.e. $cl(A)\cap cl(C)\cap (\bar X\setminus X)=\emptyset$).\\
Both relations are normal.
\end{Proposition}
\begin{proof}
Left to the reader.
\end{proof}

\begin{Observation}
Notice $X$ has its own topological orthogonality relation. However, it is not normal if the topology of $X$ is not normal.
\end{Observation}

\begin{Proposition}\label{ContinuityForLsTopologicalRelation}
Suppose $\bar X$ is a compactification of a locally compact Hausdorff space $X$. If $\perp$ is the relation defined by $A\perp C$ to mean that $cl(A)\cap cl(C)\cap (\bar X\setminus X)=\emptyset$, where closures are in $\bar X$, then $f:X\to [a,b]$
is $\perp$-continuous if and only if it extends to $\bar f:\bar X\to [a,b]$
that is topologically continuous at each point of $\bar X\setminus X$.
\end{Proposition}
\begin{proof}
Suppose $\bar f:\bar X\to [a,b]$ is topologically continuous at each point of $\bar X\setminus X$ and $A, C\subset [a,b]$ are metrically separated. Choose $\epsilon > 0$
such that $|x-y| > \epsilon$ for all $(x,y)\in A\times C$.
If there is $z\in cl(f^{-1}(A))\cap cl(f^{-1}(C))\cap (\bar X\setminus X)$,
then there is a neighborhood $U$ of $z$ in $\bar X$ such that
$\diam(f(U)) < \epsilon/3$, there is $z_1\in U\cap f^{-1}(A)$,
and there is $z_2\in U\cap f^{-1}(C)$. Hence $|f(z_1)-f(z_2)| < \epsilon$, a contradiction.

Suppose $f:X\to [a,b]$
is $\perp$-continuous. Given $z\in \bar X\setminus X$ notice that
the intersection of all sets $cl(f(U\cap X))$, $U$ a neighborhood of $z$ in $\bar X$
consist exactly of one point. Let that point be the value of $\bar f(z)$. Notice
$\bar f$ is continuous at $z$.
\end{proof}

\begin{Corollary}\label{ContinuityForSsTopologicalRelation}
Suppose $\bar X$ is a compactification of a locally compact Hausdorff space $X$. If $\perp$ is the relation defined by $A\perp C$ to mean that $cl(A)\cap cl(C)=\emptyset$, where closures are in $\bar X$, then $f:X\to [a,b]$
is $\perp$-continuous if and only if it extends to $\bar f:\bar X\to [a,b]$
that is topologically continuous at each point of $\bar X$.
\end{Corollary}
\begin{proof}
One direction is obvious.
Suppose $f:X\to [a,b]$
is $\perp$-continuous. By \ref{ContinuityForLsTopologicalRelation} it extends
over $\bar X$ to a continuous function.
\end{proof}

\section{Compatible orthogonal relations}

\begin{Definition}
Given two orthogonal relations $\perp_1$ and $\perp_2$ on a set $X$,
we define the relation $\perp_1\cap\perp_2$ as follows:
$A(\perp_1\cap\perp_2)C$ if and only if $A\perp_1 C$ and $A\perp_2 C$.
\end{Definition}

The following is obvious.
\begin{Proposition}
Given two orthogonal relations $\perp_1$ and $\perp_2$ on a set $X$,
$\perp_1\cap\perp_2$ is an orthogonal relation.
\end{Proposition}

\begin{Definition}\label{CompatibilityDef}
Suppose $X$ is a set with an orthogonal relation $\perp$.
A small scale orthogonal relation $\perp_{ss}$ on $X$ is \textbf{compatible} with $\perp$ if for every $\perp$-continuous function $f:X\to [a,b]$ and every $\epsilon > 0$
there is a $\perp$-bounded subset $B$ of $X$ and a $(\perp\cap\perp_{ss})$-continuous
function $g:X\to [a,b]$ such that $|f(x)-g(x)| < \epsilon$ for all $x\in X\setminus B$.
\end{Definition}

\begin{Lemma}[Pasting Lemma]\label{Pasting Lemma}
Suppose $X$ is a set with an orthogonal relation $\perp$
and $f:X\to [a,b]$ is a function. If $a < c < d < b$ and 
both $f|f^{-1}[a,d]:f^{-1}[a,d]\to [a,d]$ and $f|f^{-1}[c,b]:f^{-1}[c,b]\to [c,b]$
are $\perp$-continuous, then $f$ is $\perp$-continuous provided
$f^{-1}[a,c]\perp f^{-1}[d,b]$.
\end{Lemma}
\begin{proof}
Given two metrically separated subsets $A$ and $C$ of $[a,b]$
find $\delta > 0$ such that $|x-y|\ge 4\delta$ if $x\in A$ and $y\in C$. Also,
$4\delta < \min(c-a,d-c,b-d)$. Cover $[a,b]$ by finitely many intervals of length at most $\delta$, each of them contained in either $[a,c]$, $[c,d]$ or $[d,b]$
Pick among them intervals $A_i$ intersecting $A$. Pick among them intervals $C_j$ intersecting $C$.
It suffices to show that $f^{-1}(A_i)\perp f^{-1}(C_j)$ for all $i,j$.
It is so if $A_i, C_j\subset [a,d]$ or $A_i, C_j\subset [c,b]$.
The same happens if one of them is contained in $[a,c]$ and the other in $[d,b]$.
Those are the only possibilities.
\end{proof}

\begin{Proposition}\label{ACharOfCompatibility}
Suppose $X$ is a set with a normal orthogonal relation $\perp$ and a small scale orthogonal relation $\perp_{ss}$ on $X$ has the property that for every $\perp$-bounded
subset $B$ of $X$ there is a $\perp$-bounded subset $U$ containing $B$
and satisfying $(X\setminus U)\perp_{ss} B$. $\perp_{ss}$ is compatible with $\perp$ if and only if the following two conditions are satisfied:\\
1. $A\perp C$ and $A\cap C=\emptyset$ implies $(A\setminus B)\perp_{ss} (C\setminus B)$ for some $\perp$-bounded subset $B$ of $X$.\\
2. $\perp\cap\perp_{ss}$ is normal as well.
\end{Proposition}
\begin{proof}
Assume Conditions 1 and 2 hold.
It suffices to show is that for every $\perp$-continuous function
$f:X\to [-3a,3a]$ and every $\delta > 0$ there is a $(\perp\cap\perp_{ss})$-continuous and $g:X\to [-3a,3a]$ such that $|f-g|\leq 2a+2\delta$ outside of a bounded subset $B$ of $X$.

Assume $\delta < a$.
Since $f^{-1}[-a,a]$, $f^{-1}(-3a)$, $f^{-1}(3a)$
are mutually $\perp$-orthogonal, there is a $\perp$-bounded subset $B$
of $X$ such that removing $B$ from the above sets makes them
$\perp_{ss}$-orthogonal.

Using \ref{RealExtensionLemma} for $(\perp\cap\perp_{ss})$-continuity
create partial $g$ on $f^{-1}[-a,a]\setminus B$ with values in $[-a,a]$
by sending $f^{-1}(-a)\setminus B$ to $-a$, by sending $f^{-1}(a)\setminus B$
to $a$, and then extending. Notice $g$ is $(\perp\cap\perp_{ss})$-continuous.

Create another partial $g$ on $(f^{-1}[-a,-a+\delta]\cup f^{-1}(-3a)\setminus B$
with values in $[-3a,-a+\delta]$
by sending $f^{-1}(-3a)\setminus B$
to $-3a$. 

Create a third partial $g$ on $(f^{-1}[a-\delta,a]\cup f^{-1}(3a)\setminus B$
with values in $[a-\delta,3a]$
by sending $f^{-1}(3a)\setminus B$
to $3a$. 

 Paste the three extensions
using Pasting Lemma \ref{Pasting Lemma}.
Finally, extend from $X\setminus B$ over $X$.

Assume $\perp_{ss}$ is compatible with $\perp$.
Suppose $A\perp C$ and $A\cap C=\emptyset$.
The function $f:A\cup C\to [0,1]$, $f(A)\subset \{0\}$, $f(C)\subset \{1\}$
is $\perp$-continuous, so we may extend it over the whole $X$.
Let $g:X\to [0,1]$ be a $(\perp\cap\perp_{ss})$-continuous function
so that $|f(x)-g(x)| < 0.2$ for all $x\in X\setminus B$, $B$ a $\perp$-bounded subset of $X$. Notice $A\setminus B\subset g^{-1}[0,0.2]$ and $C\setminus B\subset g^{-1}[0.8,2]$
are $\perp_{ss}$-orthogonal.

Suppose $A(\perp\cap\perp_{ss}) C$. Hence $A\cap C=\emptyset$.
As above, there is a $(\perp\cap\perp_{ss})$-continuous function
$g:X\to [0,1]$ and a $\perp$-bounded subset $B$ of $X$
such that $A\setminus B\subset g^{-1}[0,0.2]$ and $C\setminus B\subset g^{-1}[0.8,2]$.
Put $A'=g^{-1}[0,0.6]$ and $C'=g^{-1}[0.4,1]$.
Thus $A'\cup C'=X$, $(A\setminus B)\perp\cap\perp_{ss} C'$,
and $(C\setminus B)\perp\cap\perp_{ss} A'$.
Pick a $\perp$-bounded $U_1$ containing $A\cap B$, $U_1\perp_{ss} C$.
Pick a $\perp$-bounded $U_2$ containing $C\cap B$, $U_2\perp_{ss} A$.
Now, $C\perp_{ss}(A'\cup U_1)$ and, since $U_1$ is $\perp$-bounded,
$C\perp(A'\cup U_1)$. Similarly, $A\perp_{ss}(C'\cup U_2)$ and, since $U_2$ is $\perp$-bounded,
$A\perp(C'\cup U_2)$. Since $A\subset (A'\cup U_1)$ and $C\subset (C'\cup U_2)$,
$\perp\cap\perp_{ss}$ is normal as well.
\end{proof}

\begin{Corollary}\label{TopologicalCompatibilityCriterion}
Suppose $\perp$ is normal orthogonal relation $\perp$ on $X$ such that self-$\perp$-orthogonal subsets of $X$ are $\perp$-bounded. Suppose $\perp_{ss}$ is the topological orthogonality relation induced by a normal topology on a set $X$.
$\perp_{ss}$ is compatible with $\perp$ if the following
conditions are satisfied:\\
1. $A\perp C$ implies $cl(A)\perp cl(C)$, where the closures are with respect to the topology on $X$,\\
2. For any $\perp$-bounded set $B$ there is an open set $U$ containing $B$ that is $\perp$-bounded.
\end{Corollary}
\begin{proof}
Suppose $A(\perp\cap\perp_{ss})C$. Therefore $cl(A)\cap cl(C)=\emptyset$
and $cl(A)\perp cl(C)$. Hence, there are subsets $A', C'$ of $X$ so that
$A'\perp C'$, $cl(A)\subset A'$, $C\subset C'$, and $A'\cup C'=X$.

Let $B_1=cl(A)\cap cl(C')$ and $B_2=cl(A')\cap cl(C)$. Both are disjoint $\perp$-bounded as they are self-$\perp$-orthogonal. Pick open $\perp$-bounded sets $U_1$ containing $B_1$ and $U_2$ containing $B_2$ whose closures are disjoint, $cl(A)\cap cl(U_2)=\emptyset$,
$cl(C)\cap cl(U_1)=\emptyset$.
Let $C'':=U_2\cup cl(C')\setminus U_1$ and $A'':= U_1\cup cl(A')\setminus U_2$.
Notice $X=A''\cup C''$, $cl(A)\perp C''$, $cl(C)\perp A''$,
$cl(A)\perp_{ss} C''$, and $cl(C)\perp_{ss} A''$.
That proves $\perp\cap\perp_{ss}$ is normal.
\end{proof}

\section{Parallelism of sets}

%beginComment
Using orthogonality relations one can define parellelism of subsets of $X$.
\begin{Definition}
$A$ is \textbf{parallel} to $C$ if $B\subset A$ and $B\perp C$ implies $B$ is bounded.
\end{Definition}

%endComment

%beginComment
\begin{Example}
In the topological case of normal spaces it means $A\subset cl(C)$.
\end{Example}

%endComment

%beginComment
\begin{Example}
In the ls-metric case it means existence or $r > 0$ such that
$$A\subset B(C,r).$$
Thus, two subsets $A, C$ of a metric space are parallel to each other if their Hausdorff distance is finite.
\end{Example}

\begin{Observation}
Two lines on the plane $\mathbb{R}^2$ are Euclidean parallel exactly when their Hausdorff distance is finite.
\end{Observation}

%endComment

%beginComment
\begin{Definition}
$f,g:X\to Y$ are \textbf{parallel} if for each $A\subset X$, $f(A)$ and $g(A)$ are parallel to each other.
\end{Definition}

\begin{Example}
In the topological case of normal spaces it means $f=g$.
\end{Example}

\begin{Example}
In the ls-metric case it means the existence of $r > 0$ such that
$$d_Y(f(x),g(x)) < r$$
for all $x\in X$.
\end{Example}

%endComment

\begin{Lemma}\label{ParallelMapsLemma}
Suppose $f:X\to Y$ is $\perp$-continuous and
 $A\subset X$ is parallel to $C\subset X$.
If $D\subset f(A)$ is orthogonal to $f(C)$, then $D=f(B)$ for some bounded subset $B$ of $X$.
\end{Lemma}
\begin{proof}
$f^{-1}(D)\perp_X f^{-1}(f(C))$ resulting in $(A\cap f^{-1}(D))\perp_X C$
as $C\subset f^{-1}(f(C))$. Thus, $B=A\cap f^{-1}(D)$ is bounded.
Notice $D=f(B)$.
\end{proof}

\begin{Corollary}\label{PerpContImpliesParallelPreserving}
Suppose $f:X\to Y$ is $\perp$-continuous and preserves bounded sets.
 If $A\subset X$ is parallel to $C\subset X$,
then $f(A)$ is parallel to $f(C)$.
\end{Corollary}

\section{Abstract boundary at infinity}

Recall that J.Roe \cite{Roe lectures} (pp.~30--31) defined the Higson corona of a coarse space $X$  as a compact space $\nu X$ satisfying
$$C(\nu X)=\frac{B_h(X)}{B_0(X)}.$$ 
Here $B_h(X)$ is the C$^\ast$-algebra of all bounded slowly oscillating complex-valued functions and $B_0(X)$ is the closed two-sided ideal of functions that 'approach $0$ at infinity', i.e. all $f \in B_h(X)$ such that for every $\epsilon >0$ the set $\{x\in X\mid |f(x)|\ge \epsilon\}$ is bounded.

In this section we generalize the concept of Higson corona to arbitrary sets equipped with an orthogonality relation.

\begin{Definition}
Given an orthogonality relation $\perp$ on subsets of $X$, a function
$f:X\to \mathbb{C}$ \textbf{$\perp$-tends to $0$ at infinity} if for every $\epsilon >0$ the set $\{x\in X\mid |f(x)|\ge \epsilon\}$ is $\perp$-bounded.

Equivalently, $f^{-1}(\mathbb{C}\setminus B(0,\epsilon))$ is $\perp$-bounded
for each $\epsilon > 0$.
\end{Definition}

\begin{Lemma}
Functions that $\perp$-tend to $0$ at infinity are $\perp$-continuous.
\end{Lemma}
\begin{proof}
Suppose $A, D\subset \mathbb{C}$ are metrically separated.
There is $\epsilon > 0$ such that $B(0,\epsilon)$ intersects at most one of the sets
$A$ and $D$. That means point-inverse of one of those sets is $\perp$-bounded
resulting in $f^{-1}(A)\perp f^{-1}(D)$.
\end{proof}

\begin{Definition}
Given an orthogonality relation $\perp$ on subsets of $X$, its \textbf{abstract boundary at infinity} $\partial X$ is the spectrum of the $C^\ast$-algebra $B^{\perp}(X)$ of $\perp$-continuous maps $f:X\to \mathbb{C}$ with bounded (in the metric sense) image modulo its two-sided ideal $B^{\perp}_0(X)$ of functions that $\perp$-tend to $0$ at infinity.
\end{Definition}

%endComment

%beginComment
\begin{Theorem}
Any $\perp$-continuous function $f:X\to Y$ induces
a continuous function from $\partial X$ to $\partial Y$. If $f$ and $g$ are parallel $\perp$-continuous functions that preserve bounded sets,
then the induced continuous functions from $\partial X$ to $\partial Y$ are equal.
\end{Theorem}
\begin{proof}
Given a $\perp_Y$-continuous map $h:Y\to \mathbb{C}$ with bounded image,
$h\circ f:  X\to \mathbb{C}$ is an
$\perp_X$-continuous map with bounded image.
Moreover, if $h$ tends to $0$ at infinity, so does $ h\circ f$.

If $f$ and $g$ are parallel $\perp$-continuous functions that preserve bounded sets and $h:Y\to \mathbb{C}$ is $\perp_Y$-continuous map with a bounded image, then 
we need to show
$h\circ f-h\circ g$ $\perp$-tends to $0$ at infinity.

Notice $h\circ f$ and $h\circ g$ are parallel by \ref{PerpContImpliesParallelPreserving}.
Suppose there is $\epsilon > 0$ such that the set
$$U:=\{x\in X\mid |h\circ f(x)-h\circ g(x)|\ge \epsilon\}$$
is unbounded. 

Cover $\{(z_1,z_2)\in h(Y)\times h(Y)\mid |z_1-z_2|\ge \epsilon\}$ by finitely many 
sets of the form $B_1\times B_2$, where each $B_i$ is an $\epsilon/8$-ball.
Since points $(h(f(x)),h(g(x)))$, $x\in U$ belong to the union of those sets,
an unbounded subset $V$ of $U$ lands in exactly one set $B_1\times B_2$.
Since $B_1\perp B_2$ in $\mathbb{C}$, $f(V)\perp g(V)$ in $Y$. However, those two sets are parallel to each other which means they are bounded
resulting in $V$ being bounded in $X$, a contradiction.
\end{proof}

%endComment

%beginComment
\subsection{Small Scale Examples}
In the case of small scale there are no bounded subsets of $X$, so we are talking about all $\perp$-continuous maps $f:X\to \mathbb{C}$ with bounded image and the abstract boundary at infinity is simply a certain compactification of $X$.
\begin{Example}
1. In the case of topological orthogonality, $\partial X$ is the \textbf{\v Cech-Stone compactification} of $X$.\\
2. In the case of uniform orthogonality, $\partial X$ is the \textbf{Samuel-Smirnov compactification} of $X$.
\end{Example}
%endComment

\section{Geometric boundary at infinity}

There are two ways to connect abstract boundary at infinity of a coarse space $X$
to its topology. One is to give sufficient conditions for the natural homomorphism  $\alpha: \frac{B^c_h(X)}{B^c_0(X)}\to  \frac{B_h(X)}{B_0(X)}$ to be an isomorphism, where $B^c_h$ and $B^c_0$ are the subalgebras of continuous functions in $B_h$ and $B_0$ respectively. 

Here are existing results in this direction:\\
1. John Roe \cite{Roe lectures} did it in the case $X$ is a paracompact space that has a uniformly bounded cover consisting of open sets,\\
2. J.Dydak and T.Weighill \cite{DW} did it for $X$ being normal both in the topological and large scale sense. Also, they assumed existence of a uniformly bounded cover of $X$ consisting of open sets.

The other way is to detect locally compact topologies on $X$ such that, when compactifying $X$ to $\gamma(X)$ using slowly oscillating continuous functions, 
the corona $\gamma(X)\setminus X$ is homeomorphic to the Higson corona of $X$.
Results in that direction can be found in \cite{DW}  and \cite{KalHon}.

In this section we generalize the above two approaches for sets with orthogonal relations.

\begin{Lemma}\label{CompatibilityViaCStarAlgebrasLemma}
Suppose $X$ is a set with a normal orthogonal relation $\perp$. A normal small scale orthogonal relation $\perp_{ss}$ on $X$ is compatible with $\perp$ if and only if  the natural homomorphism  $\alpha: \frac{B^{\perp\cap\perp_{ss}}(X)}{B_0^{\perp\cap\perp_{ss}}(X)}\to  \frac{B^{\perp}(X)}{B^{\perp}_0(X)}$ is an isomorphism.
\end{Lemma}
\begin{proof}
Assume $\alpha: \frac{B^{\perp\cap\perp_{ss}}(X)}{B_0^{\perp\cap\perp_{ss}}(X)}\to  \frac{B^{\perp}(X)}{B^{\perp}_0(X)}$ is an isomorphism.
Given $f\in B^{\perp}(X)$, there is $g\in B^{\perp\cap\perp_{ss}}(X)$
such that $f-g\in B_0^{\perp}(X)$. Therefore, if $\epsilon > 0$
and $B:=\{x\in X\mid |f(x)-g(x)|\ge \epsilon$ we get that $g$ is a required approximation of $f$ as in Definition \ref{CompatibilityDef}.

The proof in the other direction follows the standard idea of proving the Urysohn Lemma.
\end{proof}

\begin{Corollary}
Suppose $\perp$ is normal orthogonal relation $\perp$ on $X$ such that self-$\perp$-orthogonal subsets of $X$ are $\perp$-bounded. Suppose $\perp_{ss}$ is the topological orthogonality relation induced by a normal topology on a set $X$.
The natural homomorphism  $\alpha: \frac{B^{\perp\cap\perp_{ss}}(X)}{B_0^{\perp\cap\perp_{ss}}(X)}\to  \frac{B^{\perp}(X)}{B^{\perp}_0(X)}$ is an isomorphism if the following
conditions are satisfied:\\
1. $A\perp C$ implies $cl(A)\perp cl(C)$, where the closures are with respect to the topology on $X$,\\
2. For any $\perp$-bounded set $B$ there is an open set $U$ containing $B$ that is $\perp$-bounded.
\end{Corollary}

\begin{Theorem}\label{CoronaIsBoundaryAtInfinity}
Suppose $\bar X$ is a compactification of a locally compact Hausdorff space $X$,
$\perp_{ss}$ is the orthogonal relation on $X$ defined by $A\perp_{ss} C$ if and only if closures of $A$ and $C$ in $\bar X$ are disjoint,
and $\perp$ is an orthogonal relation on $X$ whose family of $\perp$-bounded subsets is identical with the family of pre-compact subset of $X$. If the natural homomorphism $\alpha: \frac{B^{\perp\cap\perp_{ss}}(X)}{B_0^{\perp\cap\perp_{ss}}(X)}\to  \frac{B^{\perp}(X)}{B^{\perp}_0(X)}$ is an isomorphism, then 
the corona $\bar X\setminus X$ is homeomorphic to the abstract boundary at infinity of $(X,\perp)$. 
\end{Theorem}
\begin{proof}
If $f:X\to [a,b]$ is $\perp$-continuous, we find $g:X\to [a,b]$
that is $(\perp\cap\perp_{ss})$-continuous and $f-g$ $\perp$-tends to $0$ at infinity.
Extend $g$ to a continuous $\bar g:\bar X\cup X\to [a,b]$.
The restriction of $\bar g$ to $\bar X\setminus X$ does not depend on $g$.
That gives an isomorphism between $C(\bar X\setminus X)$
and $\frac{B^{\perp}(X)}{B^{\perp}_0(X)}$.
\end{proof}

%beginComment
\subsection{Large Scale Examples of boundary at infinity}
The following examples follow from \ref{CoronaIsBoundaryAtInfinity} and \cite{JDEnds}.
\begin{Example}
1. In the case of hyperbolic orthogonality, $\partial X$ is the \textbf{Gromov boundary} of $X$.\\
2. In the case of metric ls-orthogonality, $\partial X$ is the \textbf{Higson corona} of $X$.\\
3. In the case of Freundenthal orthogonality, $\partial X$ is the \textbf{Freundenthal corona} of $X$.\\
4. In the case of ls-orthogonality induced by a compactification $\bar X$ of a normal locally compact space $X$, $\partial X$ is homeomorphic to $\bar X\setminus X$.
\end{Example}

%endComment

\section{Large scale compactifications}

This section is about a concept that unifies Higson compactifications, Gromov boundary, \v Cech-Stone compactification, Samuel-Smirnov compactification, and Freudenthal compactification.

\subsection{Large scale topology}
\begin{Definition}
A \textbf{large scale topological space} $(X,\mathcal{T},\mathcal{B})$ is a topological space $(X,\mathcal{T})$ in which a bornology $\mathcal{B}$
of open-closed subspaces is selected.
\end{Definition}

\begin{Definition}
A large scale topological space $(X,\mathcal{T},\mathcal{B})$ is \textbf{large scale compact} if and only if, for any family $\{U_s\}_{s\in S}$ of open subsets of $X$, $X=\bigcup\limits_{s\in S}U_s$ implies existence of a finite subset $F$ of $S$ such that
$X\setminus \bigcup\limits_{s\in F}U_s$ belongs to $\mathcal{B}$.
\end{Definition}

\begin{Proposition}\label{NormalityOfLargeScaleCompactSpaces}
1. If $(X,\mathcal{T},\mathcal{B})$ is large scale compact and Hausdorff, then it is regular.\\
2.  If $(X,\mathcal{T},\mathcal{B})$ is large scale compact and regular, then it is normal.
\end{Proposition}
\begin{proof}
1). Suppose $A$ is a closed subset of $X$ not containing $x_0$. If $x_0$ is open-closed, then $A\subset X\setminus \{x_0\}$ is disjoint from $x_0$ and we are done. Assume $x_0$ is not open.
For each point $x\in A$ choose disjoint open sets $U_x$ containing $x$ and $V_x$ containing $x_0$.
Notice $X=(X\setminus A)\cup\bigcup\limits_{x\in A}U_x$, so there is an open-closed set $B\in \mathcal{B}$ and a finite subset $F$ of $A$ such that
$X=B\cup (X\setminus A)\cup \bigcup\limits_{x\in A}U_x$. Notice $B$ does not contain $x_0$. Therefore, $A\subset B\cup \bigcup\limits_{x\in A}U_x$ is disjoint from
$\bigcap\limits_{x\in F}V_x\setminus B$ which contains $x_0$.\\
The proof of 2) is similar or apply 1) to $X/A$.
\end{proof}

\subsection{Uniqueness of large scale compactifications}

\begin{Definition}
Given an orthogonal relation $\perp$ on a set $X$, a \textbf{large scale compactification}
$\bar X$ is a large scale compact space containing $X$ as a dense subset and satisfying the following properties:\\
1. $\perp$-bounded subsets of $X$ form the selected bornology of open-closed subsets of $\bar X$,\\
2. Two subsets $C$ and $D$ of $X$ are $\perp$-orthogonal if and only if the intersection of their closures in $\bar X$ is contained in $X$ and is $\perp$-bounded.
\end{Definition}

\begin{Proposition}
If $(X,\perp)$ has a large scale compactification $\bar X$ that is Hausdorff, then $\perp$ is a normal orthogonal relation.
\end{Proposition}
\begin{proof}
By \ref{NormalityOfLargeScaleCompactSpaces}, $\bar X$ is normal. Suppose $C\perp D$
for some subsets $C, D$ of $X$. The intersection $B$ of their closures is a $\perp$-bounded subset of $X$, hence open-closed in $\bar X$.
Therefore closures of $C$ and $D\setminus B$ are disjoint in $\bar X$ and 
there exist open neighborhoods $U$ of $cl(C)$ and $V$ of $cl(D\setminus B)$ whose closures are disjoint in $\bar X$. Notice $C\perp (X\setminus U)$
and $D\perp (X\setminus V)$.
\end{proof}

\begin{Theorem}\label{FirstTheoremOnLSCompactifications}
Suppose $\bar X$ is a large scale compactification of $(X,\perp_X)$ with the family of $\perp_X$-bounded sets $\mathcal{B}_X$ and $\bar Y$ is a large scale Hausdorff compactification of $(Y,\perp_Y)$ with the family of $\perp_Y$-bounded sets $\mathcal{B}_Y$.\\
1. Any $\perp$-continuous function $f:X\to Y$ extends uniquely to a continuous function $\bar f:\bar X\to \bar Y$.\\
2. If $f:X\to Y$ satisfies $f^{-1}(B)\in \mathcal{B}_X$
for each $B\in \mathcal{B}_Y$ and extends to a continuous function 
$\bar f:\bar X\to \bar Y$, then $f$ is $\perp$-continuous.
\end{Theorem}
\begin{proof}
1. Given $x\in \bar X\setminus X$ notice the family $\{f(D\cap X)\mid D \mbox{ a neighborhood of } x \mbox{ in }\bar X\}$
consists of $\perp_Y$-unbounded sets in $Y$ and the intersection of closures in $\bar Y$ of that family is non-empty. Indeed, if $f(D\cap X)$ is $\perp$-bounded for some neighborhood $D$ of $x$ in $\bar X$, then $D\cap X\subset f^{-1}(f(D\cap X))\in \mathcal{B}_X$ is contained in $X$ contradicting $x\in \bar X\setminus X$.

If the intersection contains exactly one element, it is a good candidate for
$\bar f(x)$. If there are two different points
$y_1$ and $y_2$ in the intersection, we pick closed neighborhoods $C_1\perp_Y C_2$, $C_i$ of $y_i$ for $i=1,2$, and arrive at a contradiction. Namely, both $f^{-1}(C_1)$ and $f^{-1}(C_2)$ are disjoint $\perp_X$-orthogonal sets, so their closures in $\bar X$ intersect along a set $B\in \mathcal{B}_X$. One of the closures, say $cl(f^{-1}(C_2))$, does not contain $x$, so $D:=\bar X\setminus cl(f^{-1}(C_2))$ is a neighborhood of $x$. However, $cl(f(D\cap X))$ misses $y_2$, a contradiction.
A similar argument shows $\bar f$ is continuous. \\
2. Given $C\perp_Y D$ in $Y$, we extend them to $\perp_Y$-orthogonal zero-sets $C'$ and $D'$ using \ref{MainPropOnNormalRelations} and \ref{RealExtensionLemma}. If one of them is $\perp_Y$-bounded, it is clear
$f^{-1}(C)\perp_X f^{-1}(D)$, so assume both $C'$ and $D'$ are $\perp_Y$-unbounded.
Now, $cl(C')\cap cl(D')$ is contained in $Y$ and is $\perp_Y$-bounded.
Therefore the intersection of their point-inverses is $\perp_X$-bounded
and $f^{-1}(C)\perp_X f^{-1}(D)$.
\end{proof}

\begin{Corollary}
If $\perp$ is a normal orthogonality relation on $X$ and $\bar X$ is its large scale compactification, then $\bar X\setminus \bigcup\limits_{B\in\mathcal{B}} B$
is the abstract boundary of infinity of $(X,\perp)$, where $\mathcal{B}$ is the bornology of $\perp$-bounded subsets of $X$.
\end{Corollary}

\subsection{Existence of large scale compactifications}

\begin{Definition}
Given a set $X$ with orthogonality relation $\perp$, a subset $A$ of $X$ is called
a \textbf{$\perp$-zero-set} (\textbf{$\perp$-cozero-set}, respectively)
if there is a $\perp$-continuous function $f:X\to [0,1]$ such that
$A=f^{-1}(0)$ ($A=f^{-1}(0,1]$, respectively).
\end{Definition}

\begin{Proposition}\label{UnionOfZeroSetsProposition}
The union of two $\perp$-zero-sets ($\perp$-cozero-sets, respectively) in $(X,\perp)$ is a $\perp$-zero-set ($\perp$-cozero-set, respectively).
\end{Proposition}
\begin{proof}
Given two $\perp$-continuous functions $f,g:X\to [0,1]$ their product and their sum is a $\perp$-continuous function (apply \label{ProductOfPerpFunctions}). Look at zero-sets (cozero-sets, respectively) generated by those functions.
\end{proof}

\begin{Definition}
Given a set $X$ with orthogonality relation $\perp$ consider the family $\partial_0 X$ of all ultrafilters in $X$ consisting of $\perp$-unbounded subsets of $X$ that are $\perp$-zero-sets.

By $\bar X_0$ we mean $X\cup \partial_0 X$ with the understanding that principal ultrafilters, i.e. those containing all $\perp$-zero-supsets of $\{x\}$ for some $x\in X$, are identified with that particular point of $X$.

Given $C\subset X$, by $\bar C$ we mean the union of $C$ and of all ultrafilters in $ \partial_0 X$ containing a subset of $C$. In particular, if $C$ is $\perp$-bounded, then $\bar C=C$.
\end{Definition}

\begin{Lemma}\label{BarLemma} 
If $E=C\cap D$ and $F=C\cup D$, then $\bar E=\bar C\cap \bar D$ and
$\bar F=\bar C\cup \bar D$ if both $C$ and $D$ are $\perp$-cozero-sets in $X$.
\end{Lemma}
\begin{proof}
Any ultrafilter containing subsets of both $C$ and $D$ also contains a subset of $C\cap D$.
$\bar C\cup \bar D\subset \bar F$ is obvious. Suppose $\mathcal{F}$ is 
an ultrafilter containing a subset $H$ of $C\cup D$ but no subsets of $C$ or $D$.
Pick continuous functions $\alpha_C,\alpha_D:X\to [0,1]$ such that
$C=\alpha_C^{-1}(0,1]$ and $D=\alpha_D^{-1}(0,1]$.
Let $H_C:=\{x\in \alpha_C(x)\ge \alpha_D(x)\}$ and $H_D:=\{x\in \alpha_D(x)\ge \alpha_C(x)\}$. Both are $\perp$-zero-sets if $H$ is (the proof is similar to that of \ref{UnionOfZeroSetsProposition}).
Since $H_C\notin{F}$, there is $G_1\in\mathcal{F}$ disjoint from $H_C$.
Similarly, there is $G_2\in\mathcal{F}$ disjoint from $H_D$.
That means $G_1\cap G_2$ is disjoint from $H$, a contradiction.
\end{proof}

\begin{Theorem}
$X\cup\partial_0 X$ is large scale compact if the topology has 
$$\{\bar U\mid U \mbox{ is a $\perp$-cozero-set in } X\}$$ as its basis, and $\mathcal{B}$ is the bornology of bounded subsets of $X$.
Moreover, if $\perp$ is a normal orthogonality relation, then $X\cup\partial_0 X$
is normal.
\end{Theorem}
\begin{proof}
Suppose $\{U_s\}_{s\in S}$ is a family of $\perp$-cozero-sets of $X$ such that
$X\cup\partial_0 X=\bigcup\limits_{s\in S}\bar U_{s}$. Our goal is to show existence of a finite subset $F$ of $S$ such that $B_F:=X\setminus \bigcup\limits_{s\in F} U_s$ is $\perp$-bounded. In that case, $(X\cup\partial_0 X)-\bigcup\limits_{s\in F}\bar U_{s}$ is $\perp$-bounded and we are done.
Indeed, given $\mathcal{F}\in \partial_0 X$, there is $D\in \mathcal{F}$ contained
in $\bigcup\limits_{s\in F}U_{s}$. $D$ can be expressed as $D=\bigcup\limits_{s\in F} D_{s}$, where $D_s\subset U_s$ is a zero-set in $X$. Now, there is $t\in F$ such that $D_t$ is $\perp$-unbounded resulting in $\mathcal{F}\in \bar U_t$. 

Suppose $B_F:=X\setminus \bigcup\limits_{s\in F} U_s$ is $\perp$-unbounded
for all finite subsets $F$ of $S$. There is an ultrafilter $\mathcal{F}$ containing all those sets as they are $\perp$-zero-sets (see \ref{UnionOfZeroSetsProposition}). Also, $\mathcal{F}\in \bar U_{t}$ for some $t\in S$ which means $D\in\mathcal{F}$
for some $D\subset U_t$, contradicting
$X\setminus U_t\in \mathcal{F}$.

Suppose $\perp$ is a normal orthogonality relation. To show $X\cup\partial_0 X$
is normal, it suffices to prove it is Hausdorff in view of \ref{NormalityOfLargeScaleCompactSpaces}. Clearly it is so on the union of $\perp$-bounded subsets of $X$, so assume $\mathcal{F}_1\ne\mathcal{F}_2$.
In that case there are disjoint $\perp$-zero-sets $D_i\in\mathcal{F}_i$, $i=1,2$.
That leads to a continuous function $\alpha:X\to [1,2]$ satisfying $\alpha^{-1}(i)=D_i$ for $i=1,2$. Notice that $\bar U_1\cap \bar U_2=\emptyset$, where $U_1=\alpha^{-1}[1,1.2)$ and $U_2=\alpha^{-1}(1.7,2]$.
\end{proof}

\begin{Definition}
Given a set $X$ with orthogonality relation $\perp$ we introduce a relation on elements of $\partial_0 X$ as follows: $\mathcal{F}_1\sim \mathcal{F}_2$ if $C$ is not
$\perp$-orthogonal to $D$ whenever $C\in \mathcal{F}_1$ and $D\in \mathcal{F}_2$.
\end{Definition}

\begin{Proposition}
If $\perp$ is a normal orthogonality relation on $X$, then $\sim$ is an equivalence relation
and the quotient space under this relation, denoted by $\bar X:=X\cup\partial X$
is large scale compact and Hausdorff.
Moreover, two subsets $C$ and $D$ of $X$ are $\perp$-orthogonal if and only if
the intersection of their closures in $\bar X$ is contained in $X$ and is $\perp$-bounded.
\end{Proposition}
\begin{proof}

\textbf{Claim:} If $C$ and $D$ are disjoint, $\perp$-orthogonal $\perp$-zero-sets in $X$, then there are $\perp$-cozero-sets $U_C$, $U_D$ such that $C\subset U_C$, $D\subset U_D$, $U_C\perp D$, $U_D\perp C$, and $U_C\cup U_D=X$.\\
\textbf{Proof of Claim:} Pick two $\perp$-continuous functions $f,g:X\to [0,1]$
with $C$ being the zero-set of $f$ and $D$ being the zero set of $g$.
Define $h$ as $\frac{f}{f+g}$ and notice $h$ is $\perp$-continuous using \ref{ProductOfPerpFunctions}. Define $U_C$ as $h^{-1}[0,0.6)$ and $U_D$
as $h^{-1}(0.4,1]$. $\blacksquare$\\
Notice that we can accomplish $U_C$ and $U_D$ to be both disjoint and $\perp$-orthogonal if $U_C\cup U_D=X$ is not required.

Given two ultrafilters $\mathcal{F_i}$, $i=1,2$, such that $\mathcal{F_1}\sim \mathcal{F_2}$ is false, we can choose $C\in \mathcal{F_1}$ and $D\in \mathcal{F_2}$
that are $\perp$-orthogonal. Moreover, we may assume $C\cap D=\emptyset$ by removing their intersection which is $\perp$-bounded.
Using the Claim choose $U_C$, $U_D$ such that $C\subset U_C$, $D\subset U_D$, $U_C\perp D$, $U_D\perp C$, and $U_C\cup U_D=X$.
In that case any other ultrafilter must belong to $\bar U_C$ or $\bar U_D$
thus ensuring $\sim$ is an equivalence relation.

Now we need to show that the equivalence class $[\mathcal{F_1}]$ of each ultrafilter $\mathcal{F_1}$ is a closed subset of $X\cup\partial_0 X$. That follows from the observation right after the proof of Claim.

Next, let's show that if $U$ is a neighborhood of $[\mathcal{F}]$ in $X\cup\partial_0 X$,
then there is a neighborhood $V$ of $[\mathcal{F}]$ in $U$ such that
any equivalence class intersecting $V$ is contained in $U$.

Given any ultrafilter $\mathcal{G}$ not belonging to $U$ we can find
disjoint $\perp$-cozero sets $U_G$ and $V_G$ such that $[\mathcal{F}]\subset \bar U_G\subset U$
and $\mathcal{G}\in \bar V_G$. Now, there is a $\perp$-bounded subset $B$ of $X$
and finitely many ultrafilters $\mathcal{G}(i)$, $1\leq i\leq k$,
such that $ X\cup\partial_0 X=B\cup U\cup \bigcup\limits_{i=1}^k \bar V_{G(i)}$.
Notice $(X\cup\partial_0 X) \setminus U\subset B\cup \bigcup\limits_{i=1}^k \bar V_{G(i)}$
and $[\mathcal{F}]\subset V:=\bigcap\limits_{i=1}^k \bar U_{G(i)}\setminus B$.
Observe that any equivalence class intersecting $V$ must be contained in $U$.
That quarantees $\bar X$ is large scale compact and Hausdorff.

Suppose two subsets $C$ and $D$ of $X$ are $\perp$-orthogonal, yet
the intersection of their closures in $\bar X$ contains an equivalence class
$[\mathcal{F}]$. We may assume $C\cap D=\emptyset$ and, by extending
the characteristic function of $C$ on $C\cup D$ first over $X$, then over $X\cup\partial_0 X$, we may  
find disjoint neighborhoods
of $C$ and $D$, a contradiction.

Suppose two subsets $C$ and $D$ have
the intersection of their closures in $\bar X$ contained in $X$ and being $\perp$-bounded. By removing that intersection, we may assume closures of $C$ and $D$ are disjoint. As in the proof of \ref{NormalityOfLargeScaleCompactSpaces}, we can find neighborhoods of $C$ and $D$ that are not only disjoint but $\perp$-orthogonal. Therefore $C\perp D$.
\end{proof}

\begin{Corollary}
If $\perp$ is a normal orthogonality relation on $X$, then $(X,\perp)$ has a large scale compactification and it is unique up to a homeomorphism fixing $X$.
\end{Corollary}

\end{document}